\documentclass[12pt]{amsart}
\usepackage{amscd,amsmath,amsthm,amssymb}
\usepackage{mathtools}
\usepackage[margin=2cm]{geometry}
\usepackage{epsfig}
\usepackage{rawfonts}
\usepackage{enumerate}
\usepackage{graphics}
\usepackage{multirow}
\usepackage{xspace}
\usepackage{graphicx}
\usepackage{mathrsfs}
\usepackage{amsmath}
\usepackage{amsfonts}
\usepackage{amssymb}
\usepackage{amsthm}
\usepackage{graphicx}
\usepackage{booktabs}
\usepackage{caption}
\usepackage{listings}
\usepackage{setspace}
\usepackage[mathscr]{eucal}
\usepackage{pgfplots}
\usepackage{hyperref}
\usepackage{wrapfig}
\usepackage{floatflt,epsfig}
\usepackage{ dsfont }
\usepackage{amscd}
\usepackage{tikz-cd}
\usepackage{fancyhdr}
\usepackage[all]{xy}
\usepackage{latexsym}
\usepackage{amscd}
\usepackage{pifont}
\usepackage{subfig}
\usepackage{easyReview}
\usepackage{subfig}
\usepackage{pstricks-add}
\usepackage{pgf,tikz,pgfplots}
\pgfplotsset{compat=1.15}
\usepackage{mathrsfs}
\usetikzlibrary{arrows}
\usetikzlibrary[patterns]
 \usepackage[normalem]{ulem}
 \usepackage{multicol}

\renewcommand{\qedsymbol}{$\square$}

\def\opn#1#2{\def#1{\operatorname{#2}}} 
\opn\chara{char} \opn\length{\ell} \opn\pd{pd} \opn\rk{rk}
\opn\projdim{proj\,dim} \opn\injdim{inj\,dim} \opn\rank{rank}
\opn\depth{depth} \opn\grade{grade} \opn\height{height}
\opn\embdim{emb\,dim} \opn\codim{codim}

\opn\Tr{Tr} \opn\bigrank{big\,rank}
\opn\superheight{superheight}\opn\lcm{lcm}
\opn\trdeg{tr\,deg}
	\opn\reg{reg} \opn\lreg{lreg} \opn\ini{in} \opn\lpd{lpd}
	\opn\size{size} \opn\sdepth{sdepth}
	\opn\link{link}\opn\fdepth{fdepth}\opn\lex{lex}\opn\dist{dist}
	\opn\div{div} \opn\Div{Div} \opn\cl{cl} \opn\Cl{Cl}
	\opn\Spec{Spec} \opn\Supp{Supp} \opn\supp{supp} \opn\Sing{Sing}
	\opn\Ass{Ass} \opn\Min{Min}\opn\Mon{Mon}
	\opn\Ann{Ann} \opn\Rad{Rad} \opn\Soc{Soc}
	\opn\Im{Im} \opn\Ker{Ker} \opn\Coker{Coker} \opn\Am{Am}
	\opn\Hom{Hom} \opn\Tor{Tor} \opn\Ext{Ext} \opn\End{End}
	\opn\Aut{Aut} \opn\id{id}
	
	\opn\nat{nat}
	\opn\pff{pf}
	\opn\Pf{Pf} \opn\GL{GL} \opn\SL{SL} \opn\mod{mod} \opn\ord{ord}
	\opn\Gin{Gin} \opn\Hilb{Hilb}\opn\sort{sort}
	\opn\aff{aff} \opn
	\con{conv} \opn\relint{relint} \opn\st{st}
	\opn\lk{lk} \opn\cn{cn} \opn\core{core} \opn\vol{vol}
	\opn\link{link} \opn\star{star}\opn\lex{lex}\opn\set{set}
	\opn\gr{gr}
	
	\def\pot#1#2{#1[\kern-0.28ex[#2]\kern-0.28ex]}

	\opn\dirlim{\underrightarrow{\lim}}
	\opn\inivlim{\underleftarrow{\lim}}
	\let\to=\rightarrow
	
	\def\Implies{\ifmmode\Longrightarrow \else
		\unskip${}\Longrightarrow{}$\ignorespaces\fi}
	\def\implies{\ifmmode\Rightarrow \else
		\unskip${}\Rightarrow{}$\ignorespaces\fi}
	\def\iff{\ifmmode\Longleftrightarrow \else
		\unskip${}\Longleftrightarrow{}$\ignorespaces\fi}

	\let\:=\colon
	\let\epsilon\varepsilon
	\let\kappa=\varkappa
	\def\qed{\ifhmode\textqed\fi
		\ifmmode\ifinner\quad\qedsymbol\else\dispqed\fi\fi}
	\def\textqed{\unskip\nobreak\penalty50
		\hskip2em\hbox{}\nobreak\hfil\qedsymbol
		\parfillskip=0pt \finalhyphendemerits=0}
	\def\dispqed{\rlap{\qquad\qedsymbol}}
	
	\opn\dis{dis}
	\def\pnt{{\raise0.5mm\hbox{\large\bf.}}}
	
	\opn\Lex{Lex}
	
        \newtheorem{Theorem}{Theorem}[section]
	\newtheorem{Lemma}[Theorem]{Lemma}
	\newtheorem{Corollary}[Theorem]{Corollary}
	
	\newtheorem{Remark}[Theorem]{Remark}
	
	\newtheorem{Example}[Theorem]{Example}

	\newtheorem{Question}[Theorem]{Question}

       \newtheorem{Setup}[Theorem]{Setup}
       \newtheorem{Construction}[Theorem]{Construction}

\begin{document}

        \title[Binomial ideals]{Prime and Cohen--Macaulay binomial ideals with linear resolution}

\author[T.~Hibi]{Takayuki Hibi}
\address[Takayuki Hibi]
{Department of Pure and Applied Mathematics, 
Graduate School of Information Science and Technology, 
Osaka University, 
Suita, Osaka 565-0871, Japan}
\email{hibi@math.sci.osaka-u.ac.jp}

\author[A. A. Qureshi]{Ayesha Asloob Qureshi}      
   \address[Ayesha Asloob Qureshi]{Sabanci University, Faculty of Engineering and Natural Sciences, Orta Mahalle, Tuzla 34956, Istanbul, Turkey}	
\email{aqureshi@sabanciuniv.edu, ayesha.asloob@sabanciuniv.edu}

\author[S.~Saeedi~Madani]{Sara Saeedi Madani}
\address[Sara Saeedi Madani]
{Department of Mathematics and Computer Science, Amirkabir University of Technology, Tehran, Iran, and School of Mathematics, Institute for Research in Fundamental Sciences, Tehran, Iran} 
\email{sarasaeedi@aut.ac.ir, sarasaeedim@gmail.com}

	
    
	\thanks{} 

        \maketitle
      
\begin{abstract}
Let $I$ be an ideal of a polynomial ring $S$ over a field $K$ for which $I$ is minimally generated by at least two quadratic binomials.  Suppose that (i) $I$ is prime, (ii) $S/I$ is Cohen--Macaulay and (iii) $I$ has linear resolution.  The question whether $I$ is equal to the ideal of $2$-minors of a $(2 \times n)$-matrix of variables is mainly studied.
\end{abstract}

\section*{Introduction}
Ideals generated by quadratic binomials have been studied for quarter century by a huge number of research papers and, together with monomial ideals and binomial ideals, belong to current trends on combinatorics and commutative algebra. Historically, one of the most important binomial ideals is the ideal $J$ of $2$-minors of a $(2 \times n)$-matrix of variables, viz., 
\begin{eqnarray}
\label{determinantal}
    J=(x_iy_j - x_jy_i : 1 \leq i < j \leq n).
\end{eqnarray}
Let $R=K[x_1, \ldots, x_n,y_1 \ldots, y_n]$ denote the polynomial ring in $2n$ variables over a field $K$. It is well known \cite{BV} that (i) $J$ is prime, (ii) $R/J$ is Cohen--Macaulay and (iii) $J$ has linear resolution.   It was shown in  \cite{EHH} that, in the class of binomial ideals arising from finite distributive lattices \cite{Hibi}, simple polyominoes  \cite{Q} and finite graphs \cite{HHHKR, Sara}, the ideals satisfying (i), (ii) and (iii) are exactly the ideal $J$ in (\ref{determinantal}). The observation naturally brings us to propose the following   

\begin{Question}
    \label{main question}
Let $I$ be an ideal of a polynomial ring $S$ over a field for which $I$ is minimally generated by at least two squarefree quadratic binomials.  Suppose that (i) $I$ is prime, (ii) $S/I$ is Cohen--Macaulay and (iii) $I$ has linear resolution.  Is it true that $I$ is equal to the ideal $J$ in (\ref{determinantal})?  
\end{Question}

If we remove the assumption of squarefreeness in Question~\ref{main question}, then the ideal of $2$-minors of a $3 \times 3$ symmetric matrix of variables and the ideal of $2$-minors of a $2 \times n$ Hankel matrix of variables, provides counterexamples, for example see \cite{BCV,CMVS}.  Now, considering these examples, it is reasonable to propose the following 

\begin{Question}
   \label{second question}
Let $I$ be an ideal of a polynomial ring $S$ over a field for which $I$ is minimally generated by at least two quadratic binomials.  Suppose that (i) $I$ is prime, (ii) $S/I$ is Cohen--Macaulay and (iii) $I$ has linear resolution.  Is it true that $I$ is the ideal of all 2-minors of a matrix of (not necessarily distinct) variables?  
\end{Question}

Furthermore, 

\begin{Question}
   \label{third question}
Let $X$ be a $(2 \times n)$--matrix of (not necessarily distinct) variables belonging to a polynomial ring $S$ over a field and $I$ the ideal of $S$ generated by all $2$-minors of $X$.  We call $X$ a {\em PCML} matrix if (i) $I$ is prime, (ii) $S/I$ is Cohen--Macaulay and (iii) $I$ has linear resolution.  Is it possible to classify all PCML matrices? 
\end{Question}

In the present paper, in Section $1$, we show that the answer of Question \ref{main question} and Question~\ref{second question} is ``yes'', when $I$ is minimally generated by three quadratic (squarefree) binomials.  In Section $2$, we study some general properties of binomial ideals satisfying conditions of Question~\ref{second question} which have a squarefree quadratic initial ideal. In Section~\ref{mu=6}, we focus on Question~\ref{main question} for ideals minimally generated by six squarefree quadratic binomials which also admit a squarefree quadratic initial ideal.  
On the other hand, we will not address Question~\ref{third question}, which is reserved for our forthcoming study.  Question \ref{third question} is related with \cite[Conjecture 3.3]{HH2012}.  


\section{Binomial ideals generated by three quadratic (squarefree) binomials}

The first step to study Question \ref{main question} and Question~\ref{second question} is the following

\begin{Theorem}
    \label{ayesha}
Let $I$ be an ideal of a polynomial ring $S$ for which $I$ is minimally generated by three quadratic binomials.  Suppose that (i) $I$ is prime, (ii) $S/I$ is Cohen--Macaulay and (iii) $I$ has linear resolution. Then $I$ is generated by all 2--minors of a matrix of (not necessarily distinct) variables.  

If in addition $I$ is minimally generated by three quadratic squarefree binomials, then $I$ is equal to the ideal $J$ in  (\ref{determinantal}) with $n=3$. In other words,
	\[
I=(x_1y_2-x_2y_1,\; x_1y_3-x_3y_1,\; x_2y_3-x_3y_2)
\subset K[x_1,x_2,x_3,y_1,y_2,y_3].
\]
\end{Theorem}

\begin{proof}
We first show that $\height(I)=2$. Note that $\height(I)\neq 1$, since $I$ is prime and not principal.   
Suppose that $\height(I)=3$. As $I$ is minimally generated by three elements, it follows that $I$ is a complete intersection. Then, the minimal graded free resolution of $S/I$ is given by the Koszul complex which gives $\reg(S/I)=3$, a contradiction to the assumption that $I$ has linear resolution and $\reg(I)=2$. Thus $\height(I)\neq 3$, and we conclude that $\height(I)=2$. 
	
	Since $S/I$ is Cohen--Macaulay, we have $\depth(S/I)=\dim S-2.$	By Auslander--Buchsbaum formula, $\pd(S/I)=\dim S-\depth(S/I)=2$.  It follows that $I$ admits a linear minimal free resolution of the form
	\[
	0\longrightarrow S(-3)^r \xrightarrow{\ \Phi\ } S(-2)^3 \longrightarrow I \longrightarrow 0.
	\]
By  Hilbert--Burch theorem, \cite[Theorem 20.12]{Eisenbud}, we have $r=2$ and $I=aJ$  where $J$ is generated by the 2-minors of the $2\times 3$ matrix  representing $\Phi$, and $a\in S$ is a nonzero divisor. Since $I$ is prime, $a$ must be a unit  in $K$.  Moreover,  the entries of $\Phi$ are linear forms, as $I$ has linear resolution. 

 It only remains to show that the entries of $\Phi$ are distinct variables.  It follows from \cite[Comment 5.9 (1)]{PS}, that the minimal free resolution of $I$ is monomial, that is, the entries of $\Phi$ are monomials. Therefore, the entries of $\Phi$ are variables. If, in addition, $I$ is generated by quadratic squarefree binomials, then it follows that the entries of $\Phi$ are distinct. This completes the proof.
 \, \, \, \, 
\end{proof}

By virtue of \cite[Theorem 6.1 (ii)]{PS}, in Theorem \ref{ayesha-2}, one can drop the assumption that $I$ is Cohen--Macaulay.  In fact,

\begin{Theorem}
	\label{ayesha-2}
Let $I$ be an ideal of a polynomial ring $S$ for which $I$ is minimally generated by three quadratic binomials.  Suppose that (i) $I$ is prime and (ii) $I$ has linear resolution. Then $I$ is generated by all 2--minors of a matrix of (not necessarily distinct) variables.  

If in addition $I$ is minimally generated by three quadratic squarefree binomials, then $I$ is equal to the ideal $J$ in  (\ref{determinantal}) with $n=3$. In other words,
	\[
I=(x_1y_2-x_2y_1,\; x_1y_3-x_3y_1,\; x_2y_3-x_3y_2)
\subset K[x_1,x_2,x_3,y_1,y_2,y_3].
\]
\end{Theorem}

\begin{proof}

    As in the first paragraph of the proof of Theorem \ref{ayesha}, one has $\height(I)=2$.  It then follows from \cite[Theorem 6.1 (ii)]{PS} that $S/I$ is Cohen--Macaulay and its minimal free resolution is the Hilbert--Burch complex on the minimal generators of $I$. 
	%
    %
    More precisely, the  minimal free resolution of $I$ has the form
	\[
	0\longrightarrow S^2 \xrightarrow{\ \Phi\ } S^3 \longrightarrow I \longrightarrow 0.
	\]
%
%
    Now, the final paragraph in the proof of Theorem \ref{ayesha} remains valid without modification.
    \, \, \, 
\end{proof}

Let $I$ be a graded ideal in the polynomial ring $S$. We denote the number of elements in a homogeneous minimal system of generators of $I$ by $\mu(I)$. 

We conclude this section with the following remark. 

\begin{Remark}{\em
By \cite[Theorem 6.1]{PS}, a height~2 graded binomial ideal is Cohen--Macaulay if and only if it has at most three minimal generators. 
If $\mu(I)=2$, then $S/I$ is a complete intersection and its minimal free resolution is given by the Koszul complex, but this is linear only in the case when $I$ is generated by linear forms. 
If $\mu(I)=3$, then $S/I$ is Cohen--Macaulay but not a complete intersection, and when $I$ is minimally generated by three quadratic binomials, it has a linear resolution (see \cite[Comment~5.9(g)]{PS}). 
Therefore, if $I$ is a height~2 quadratic binomial ideal such that $S/I$ is Cohen--Macaulay and $I$ has a linear resolution, then $I$ is minimally generated by three elements.  
If, in addition, $I$ is prime, then by Theorem~\ref{ayesha} $I$ is generated by all 2-minors of a matrix of variables. Furthermore, if $I$ is prime and generated by quadratic squarefree binomials, then again by Theorem~\ref{ayesha} $I$ coincides with the ideal $J$ in (\ref{determinantal}) with $n=3$.
}
\end{Remark}
\section{Binomial ideals with a squarefree quadratic initial ideal}\label{squarefree initial ideal section}

Motivated by the fact that the determinantal ideal $J$ in  (\ref{determinantal}) has a squarefree quadratic initial ideal, in this section we consider binomial ideals satisfying the conditions stated in Question~\ref{second question} that also have squarefree quadratic initial ideals. First we observe how the interesting properties of being Cohen--Macaulay and of having linear resolution for a quadratic binomial ideal and its squarefree quadratic initial ideal are related.     

\begin{Lemma}\label{equivalency of CM+linear}
Let $I$ be a quadratic binomial ideal in a polynomial ring $S$ with a squarefree quadratic initial ideal $\ini_< (I)$ with respect to some monomial order $<$. Then, the following statements hold:
\begin{enumerate}
    \item $S/I$ is Cohen-Macaulay if and only if $S/\ini_< (I)$ is Cohen-Macaulay.
    \item $I$ has linear resolution if and only if $\ini_< (I)$ has linear resolution.
    \item If $\ini_< (I)$ has linear resolution, then the graded Betti numbers of $I$ and $\ini_< (I)$ coincide, and in particular $\mu(I)=\mu(\ini_< (I))$. 
\end{enumerate}
\end{Lemma}

\begin{proof}
(1) If $S/\ini_< (I)$ is Cohen--Macaulay, then $S/I$ is Cohen--Macaulay as well \cite[Corollary~3.3.5]{HHgtm260}. Conversely, assume that $S/I$ is Cohen--Macaulay. Thus, $\pd(S/I)=\height(I)$. Since $\ini_< (I)$ is squarefree, it follows from \cite[Corollary~2.7]{CV} that $\pd (S/I)=\pd (S/\ini_< (I))$. Hence, $\pd(S/\ini_< (I))=\height(I)$. On the other hand, one has $\dim(S/I)=\dim(S/\ini_< (I))$ \cite[Theorem~3.3.4]{HHgtm260}, and hence $\height(I)=\height(\ini_< (I))$. Therefore, 
$\pd(S/\ini_< (I))=\height(\ini_< (I))$, and hence $S/\ini_< (I)$ is also Cohen--Macaulay. Thus, (1) follows.    

(2) If $\ini_< (I)$ has linear resolution, then $I$ has linear resolution as well \cite[Theorem~3.3.4]{HHgtm260}. Conversely, suppose that $I$ has linear resolution. Since $\ini_< (I)$ is squarefree, it follows from \cite[Corollary~2.7]{CV} that $\reg (I)=\reg (\ini_< (I))$. Thus, $\reg(\ini_< (I))=2$, since $I$ is quadratic and has linear resolution. Since $\ini_< (I)$ is generated in degree~2, it follows that $\ini_< (I)$ has linear resolution, and hence (2) follows.   

(3) Note that the minimal graded free resolution of $I$ is obtained by a sequence of consecutive cancellations from the minimal graded free resolution of $\ini_< (I)$ \cite[Theorem~22.12]{P}. Thus, if $\ini_< (I)$ has linear resolution, then there is no chance for any cancellation. Therefore, all the graded Betti numbers of $I$ and $\ini_< (I)$ coincide, and hence $\mu(I)=\beta_0(I)=\beta_0(\ini_< (I))=\mu(\ini_< (I))$. 
\end{proof}


Let $G$ be a finite simple graph with the vertex set $V(G)$ and the edge set $E(G)$. The complementary graph $\overline{G}$ is a simple graph with the same vertices as $G$ whose edges are exactly the non--edges of $G$. An \emph{induced subgraph} of $G$ on $W\subseteq V(G)$ is a graph on the vertex set $W$ whose edges are the edges of $G$ contained in $W$. A \emph{chordal graph} is a graph whose all induced cycles have length~3. A graph is called \emph{co--chordal} if its complementary graph is chordal. A subset $T\subseteq V(G)$ is called an \emph{independent} set if it does not contain any edges of $G$. The maximum size of an independent set in $G$ is called the \emph{independence number} of $G$ and we denote it by $\alpha(G)$. 

Also recall that the \emph{edge ideal} $I(G)$ of $G$ in a polynomial ring over a field $K$ is the quadratic squarefree monomial ideal whose unique minimal monomial generating set is corresponding to the edge set of $G$. 

\begin{Lemma}\label{initial ideal as edge ideal}
Let $I$ be an ideal in a polynomial ring $S$ minimally generated  by at least two quadratic binomials. Suppose that (i) $I$ is prime, (ii) $I$ has a squarefree quadratic initial ideal with respect to a monomial order $<$ and (iii) $I$ has linear resolution. Then, there is a finite simple graph $G$ whose vertex set is contained in the set of variables of $S$ such that 
 \begin{enumerate}
     \item $\ini_<(I)=I(G)$; 
     \item $S/I(G)$ is Cohen--Macaulay;
     \item $I(G)$ has linear resolution;
     \item $G$ is connected and co--chordal; 
     \item $|E(G)| = {V(G)-\alpha(G)+1 \choose 2}$. 
 \end{enumerate}
\end{Lemma}

\begin{proof}
  By assumption~(ii) it is clear that there is a finite simple graph $G$ whose vertex set is contained in the set of variables of $S$ and we have~(1), namely $\ini_<(I)=I(G)$. 
  
  (2) Since a prime binomial ideal is a toric ideal \cite{ES} and since the Stanley--Reisner complex of the radical of an initial ideal of a toric ideal is homeomorphic to a ball \cite{Sturmfels}, it follows that the radical of an initial ideal of a toric ideal is Cohen--Macaulay \cite{Reisner}. In particular, a squarefree initial ideal of a prime binomial ideal is Cohen--Macaulay. Thus, $S/I(G)$ is Cohen--Macaulay. 

  (3) follows from Lemma~\ref{equivalency of CM+linear}. 

  (4) If $G$ has more than one nontrivial connected components, then $I(G)$ has relations of degree~4, 
  which is a contradiction to~(3). Thus, $G$ has exactly one nontrivial connected component. On the other hand, since isolated vertices have no effect on the generators of $I(G)$, we may consider $G$ without isolated vertices. Moreover, (3) together with Fr\"oberg theorem~\cite[Theorem~9.2.3]{HHgtm260} imply that $G$ is co--chordal.

  (5) follows from (2), (3) and \cite[Corollary~2]{HR}. 
   \, \, \, \, \, \, \, \, \, \, \, \, \, \, \, \, \, \, \, \, \, \, \, \, \, \, \, \,  \, \, \, \, 
\end{proof}

We would like to remark that according to ~\cite{HR} a finite simple graph $G$ for which conditions~(2) and~(3) in Lemma~\ref{initial ideal as edge ideal} hold is called a \emph{bi-Cohen--Macaulay} graph over the field $K$, or just bi-Cohen--Macaulay if it does not depend on the field.

\medskip
Note that it also follows from Lemma~\ref{initial ideal as edge ideal}~(2) and Lemma~\ref{equivalency of CM+linear}~(1) that a prime binomial ideal with a squarefree initial ideal is Cohen--Macaulay. Therefore, for such ideals we may replace ``Cohen--Macaulay" among the assumptions of Question~\ref{second question} by ``admitting a squarefree initial ideal". In the following, we discuss the possible number of minimal homogeneous generators. 

\begin{Theorem}\label{possible number of generators}
Let $I$ be an ideal in a polynomial ring $S$ minimally generated by at least two quadratic binomials. Suppose that (i) $I$ is prime, (ii) $I$ has a squarefree quadratic initial ideal and (iii) $I$ has linear resolution. Then
\[
\mu(I) = \mu(J) = {n\choose 2},
\]
where $J$ is a determinantal ideal of the form (\ref{determinantal}) for some $n\geq 3$. 
\end{Theorem}

\begin{proof}
It follows from (\ref{determinantal}) that $\mu(J) = {n\choose 2}$ for any $n\geq 3$. Let $<$ be a monomial order such that $\ini_<(I)$ is a squarefree quadratic monomial ideal in $S$. By Lemma~\ref{initial ideal as edge ideal}, we have $\ini_<(I)=I(G)$ for a finite connected simple graph $G$ whose vertex set is contained in the set of variables of $S$. Then, by Lemma~\ref{equivalency of CM+linear}~(3) we have $\mu(\ini_<(I))=|E(G)|\geq 2$, since $\mu(I)\geq 2$. Therefore, it follows from Lemma~\ref{initial ideal as edge ideal}~(5) that 
\[
\mu(\ini_<(I))= {|V(G)|-\alpha(G)+1 \choose 2}.
\]
Let $n=|V(G)|-\alpha(G)+1$. Then, $n\geq 3$ and by Lemma~\ref{equivalency of CM+linear}~(3), it follows that $\mu(I)={n\choose 2}$, as desired.  
\, \, \, \, \, \, \, \, \, \, \, \, \, \, \, \, \, \, \, \, \, \, \, \, \, \, \, \, \, \, \, \, \, \, \, \, \, \, \, \, \, \, \, \, \, \, \, \, \, \, \, \, \, \, \, \, \, \, \, \, 
\end{proof}

\section{Ideals generated by six quadratic squarefree binomials with a squarefree quadratic initial ideal} \label{mu=6}

In Theorem~\ref{ayesha} we discussed a quadratic squarefree binomial ideal $I$ with $\mu(I)=3$ and showed that $I$ is just the determinantal ideal $J$ as in (\ref{determinantal}). According to Theorem~\ref{possible number of generators}, if in addition $I$ admits a quadratic squarefree initial ideal, then the next possible value for $\mu(I)$ is~6. To study this case, we begin with the following lemma and setup.

\begin{Lemma}\label{lem:sqfree}
Let $I$ be an ideal of $S$ generated by squarefree quadratic binomials. Then every quadratic binomial in $I$ is squarefree. 
\end{Lemma}

\begin{proof}
Suppose $I$ is generated by squarefree quadratic binomials $f_1,\dots,f_r$, and that $x_i^2-m\in I$ for some quadratic monomial $m\neq x_i^2$. Then $x_i^2-m = \sum_{j=1}^r g_j f_j,$ for some  $g_1, \ldots, g_r\in S.$
Consider the $K$-algebra homomorphism $\pi:S\to K[x_i]$ with $\pi(x_i)=x_i$ and $\pi(x_\ell)=0$ for $\ell\neq i$. Each $f_j$ involves $x_i$ with degree at most $1$, so $\pi(f_j)=0$. Hence $0=\pi\!\left(\sum_j g_j f_j\right)=\pi(x_i^2-m)=x_i^2,$
a contradiction. Thus no such binomial lies in $I$. \hspace{8cm}
\end{proof}

We next describe the framework for the case $\mu(I)=6$.

 \begin{Setup}\label{setup}{\em 
    Let $I \subset S$ be a binomial prime ideal minimally generated by six quadratic squarefree binomials. 
    Suppose there exists a monomial order $<$ such that the initial ideal $\ini_<(I)$ of $I$ with respect to $<$ is generated by quadratic squarefree monomials. 
    Let $A$ denote the reduced Gröbner basis of $I$ with respect to $<$. Then $\ini_<(I)$ can be viewed as the edge ideal of a connected graph $G$, where the vertices of $G$ are identified with some variables of $S$, as mentioned in Lemma~\ref{initial ideal as edge ideal}. 
 
   Assume that $x_1x_2, x_2x_3, x_1x_3 \in \ini_<(I) $, that is, $\{x_1, x_2\}, \{x_1, x_3\}, \{x_2, x_3\} \in E(G)$. 
    It follows that 
    \[
    B_{12} = x_1x_2 - ab, \quad B_{23} = x_2x_3 - cd, \quad B_{13} = x_1x_3 - ef
    \]
    belong to $A$ for some variables $a,b,c,d,e,f \in S$ (not necessarily distinct). Primality of $I$ implies 
    \begin{equation}\label{eq:notin}
      a,b \notin \{x_1, x_2\}, \quad c,d \notin \{x_2, x_3\}, \quad e,f \notin \{x_1, x_3\}.  
    \end{equation}
 Since $B_{12}, B_{23}, B_{13}$ are part of the reduced Gröbner basis of $I$, it follows that 
    \begin{equation}\label{eq:red1}
    ab, \; cd, \; ef \notin \ini_<(I),
    \end{equation}
    and Lemma~\ref{lem:sqfree} gives
    \begin{equation}\label{eq:neq}
    a\neq b, \; c\neq d, \; e\neq f.
    \end{equation}
    Moreover, each of the $S$-pairs
    \[
    S(B_{12}, B_{23}) = x_1cd - x_3ab, \quad 
    S(B_{12}, B_{13}) = x_2ef - x_3ab, \quad 
    S(B_{23}, B_{13}) = x_2ef - x_1cd
    \]
    reduces to $0$ with respect to $A$. Note that the monomials $ab, cd, ef$ are pairwise distinct. Indeed, for example, if $ab=cd$, then $B_{12}-B_{23}= x_2(x_1-x_3) \in I$, a contradiction.
    
    Note that at least two of the monomials $x_1cd, x_2ef, x_3ab$ belong to $\ini_<(I)$, because each of the $S$-pairs given above is in $I$.  Without loss of generality, assume that $x_1cd, x_2ef \in \ini_<(I)$. 
    Since $cd, ef \notin \ini_<(I)$, we may further assume that 
    \begin{equation}\label{eq:red2}
    x_1c, \quad x_2e \in \ini_<(I), \quad \text{with } x_1 \neq c, \; x_2 \neq e,
    \end{equation}
    because $\ini_<(I)$ is squarefree.  

For any vertex $v$ in $G$, we denote the neighborhood of $v$ in $G$ by $N_G(v)$, and we set $\deg_G(v)=|N_G(v)|$}. 
\end{Setup}

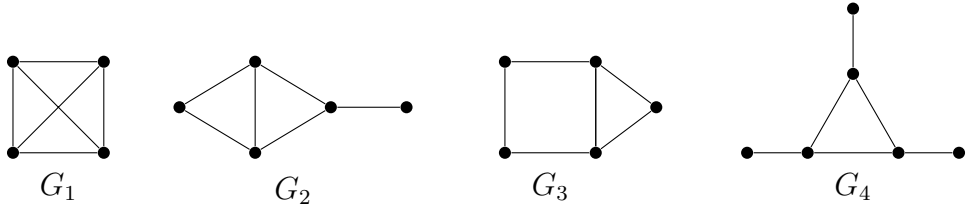
\begin{figure}[h]
\centering
\begin{tikzpicture}[scale=1]
  \tikzset{v/.style={circle, fill=black, inner sep=1.6pt}}

  \begin{scope}[shift={(0,0)}]
    \node[v] (a) at (0,0) {};
    \node[v] (b) at (1.2,0) {};
    \node[v] (c) at (1.2,1.2) {};
    \node[v] (d) at (0,1.2) {};
    \draw (a)--(b)--(c)--(d)--(a);
    \draw (a)--(c) (b)--(d);
    \node[draw=none, fill=none] at (0.6,-0.45) {$G_1$};
  \end{scope}

  \begin{scope}[shift={(3.2,0)}]
    \node[v] (b) at (0,0) {};
    \node[v] (c) at (0,1.2) {};
    \node[v] (a) at (-1.0,0.6) {};
    \node[v] (d) at (1.0,0.6) {};
    \node[v] (e) at (2.0,0.6) {};
    \draw (a)--(b)--(c)--(a);
    \draw (b)--(d)--(c);
    \draw (d)--(e);
    \node[draw=none, fill=none] at (0.5,-0.5) {$G_2$};
  \end{scope}

  \begin{scope}[shift={(6.5,0)}]
    \node[v] (d) at (0,0) {};
    \node[v] (c) at (0,1.2) {};
    \node[v] (b) at (1.2,1.2) {};
    \node[v] (a) at (1.2,0) {};
    \draw (d)--(c)--(b)--(a)--(d);
    \node[v] (e) at (2.0,0.6) {};
    \draw (a)--(b)--(e)--(a);
    \node[draw=none, fill=none] at (0.6,-0.45) {$G_3$};
  \end{scope}

  \begin{scope}[shift={(10.5,0)}]
    \node[v] (p)  at (0,0) {};
    \node[v] (q)  at (1.2,0) {};
    \node[v] (r)  at (0.6,1.04) {};
    \draw (p)--(q)--(r)--(p);
    \node[v] (p1) at (-0.8,0) {};
    \node[v] (q1) at (2.0,0) {};
    \node[v] (r1) at (0.6,1.9) {};
    \draw (p)--(p1);
    \draw (q)--(q1);
    \draw (r)--(r1);
    \node[draw=none, fill=none] at (0.6,-0.45) {$G_4$};
  \end{scope}
\end{tikzpicture}
\caption{Graphs $G_1, G_2, G_3,$ and $G_4$}
\label{badgraphs}
\end{figure}

\begin{Lemma}\label{lemma:G1G2}
With the notation introduced in Setup~\ref{setup}, the graph $G$ is not isomorphic to either of the graphs $G_1$ or $G_2$ shown in Figure~\ref{badgraphs}.
\end{Lemma}
\begin{proof}
 If $|V(G)| = 6$, then the claim holds trivially. 
Assume that $|V(G)| \leq 5$ and that $G$ is isomorphic to $G_1$ or $G_2$. 
Then each vertex in $G$ has degree at most $3$. 
The following observations hold:

\begin{enumerate}
    \item[(i)] If $G \cong G_1$, then $c = e$ and $\{x_3,c\} \in E(G)$.
    \item[(ii)] If $G \cong G_2$, then either $c = e$, or else $c \neq e$ and exactly one of the vertices $c$ and $e$ is adjacent to $x_3$ in $G$. 
\end{enumerate}

First, suppose that $c = e$. 
Then  
\[
S(B_{23}, B_{13}) = x_2cf - x_1cd = c(x_2f - x_1d).
\]
By the primality of $I$ we obtain $x_2f - x_1d \in I$.  Due to Lemma~\ref{lem:sqfree} we have $x_2 \neq f$ and $x_1\neq d$. Then, together with (\ref{eq:neq}), (\ref{eq:notin}) and the assumption $c=e$, we see that $f,d\notin\{ x_1, x_2,x_3,c\}$. This shows that if $\ini_<(x_2f - x_1d)=x_2f$, then $\deg_G(x_2) >3$, and if $\ini_<(x_2f - x_1d)=x_1d$, then $\deg_G(x_1) >3$, a contradiction.  Therefore, our assumption $c = e$ is false, and we conclude that $c\neq e$.

Following (i) above, it follows that $G$ is not isomorphic to $G_1$. 
In particular, the vertices $\{x_1, x_2, x_3, c, e\}$ are distinct in $G$, and $G \cong G_2$. By (ii), either $\{x_3,c\}$ or $\{x_3,e\}$ is an edge in $G$. 
It suffices to consider the case $\{x_3,e\} \in E(G)$. 
Then $x_3e \in \ini_<(I)$, and there exist $\ell, m \notin \{x_3,e\}$ such that 
\[
H = x_3e - \ell m \in A,  \qquad \ell m \notin \ini_<(I).
\]
Consider the $S$-pair $S(B_{23},H) = x_2\ell m - ecd.$ Since $cd \notin \ini_<(I)$ by (\ref{eq:red1}), and $ec,ed \notin \ini_<(I)$ because $\{e,c\} \notin E(G)$ and $d \notin \{x_2,x_3\}$ (see (\ref{eq:notin})), the initial term of this $S$-pair is $x_2\ell m$. 
As $\ell m \notin \ini_<(I)$, we may assume $x_2\ell \in \ini_<(I)$. 
Thus $\ell = x_1$, since $\ell \notin \{x_3,e\}$ but $\ell \in N_G(x_2) = \{x_1, x_3, e\}$. Now $H = x_3e - x_1m$, and consider the $S$-pair $S(B_{13}, H) = x_1^2m - e^2f.$ Note that $ef,\, e^2,\, x_1^2 \notin \ini_<(I)$, and $\ell m = x_1m \notin \ini_<(I)$. 
Hence, the initial term of this $S$-pair does not belong to $\ini_<(I)$, a contradiction. 
Therefore, $G$ is not isomorphic to $G_2$.
This completes the proof.
\, \, \, \, \, \, \, \, \, 
\end{proof}

\begin{Lemma}\label{lemma:G_3}
With the notation introduced in Setup~\ref{setup}, the graph $G$ is not isomorphic to the graph $G_3$ displayed in Figure~\ref{badgraphs}.
\end{Lemma}

\begin{proof}
Assume that $G \cong G_3$. 
Then $V(G) = \{x_1, x_2, x_3, c, e\}$, and together with (\ref{eq:red2}) we obtain 
\[
E(G) = \big{\{} \{x_1, x_2\}, \{x_2, x_3\}, \{x_1, x_3\}, \{x_1, c\}, \{c, e\}, \{x_2, e\} \big{\}}.
\]
Moreover, $x_3ab \notin \ini_<(I)$, by (\ref{eq:notin}) and (\ref{eq:red1}). Since $x_2e \in \ini_<(I)$, there exist $r,s \notin \{x_2, e\}$ such that
\[
F = x_2e - rs \in A,  \qquad  rs \notin \ini_<(I).
\] 
Also, since $x_1c \in \ini_<(I)$, there exist $t,u \notin \{x_1, c\}$ such that 
\[
G = x_1c - tu \in A, \qquad tu \notin \ini_<(I).
\]

Now consider 
\[
S(B_{12}, B_{13}) = x_2ef - x_3ab 
     = f(x_2e - rs) + frs - x_3ab \in I,
\]
which implies $frs - x_3ab \in I$ with $\ini_<(frs - x_3ab) = frs \in \ini_<(I)$. 
Since $rs \notin \ini_<(I)$, 
either $fr$ or $fs$ belongs to $\ini_<(I)$. In particular, $f\in V(G)$. On the other hand, we have $f \notin \{x_1, x_3, e\}$ by (\ref{eq:notin}) and (\ref{eq:neq}); $f \neq x_2$ because $fe \notin \ini_<(I)$ but $x_2e \in \ini_<(I)$ (see (\ref{eq:red1}) and (\ref{eq:red2})); $f \neq c$ 
because $fe \notin \ini_<(I)$ while $ec \in \ini_<(I)$.  This contradicts $f\in V(G)$, and the proof is complete. 
\, \, \, \, \, \, \, \, \, \, \, \, \, 
\end{proof}

\begin{Lemma}\label{lemma:G_4}
With the notation introduced in Setup~\ref{setup}, the graph $G$ is not isomorphic to the graph $G_4$ displayed in Figure~\ref{badgraphs}.
\end{Lemma}

\begin{proof}
Assume that $G \cong G_4$. 
Then $V(G) = \{x_1, x_2, x_3, c, e, \ell\}$, and together with (\ref{eq:red2}) we obtain 
\[
E(G) = \big{\{} \{x_1, x_2\}, \{x_2, x_3\}, \{x_1, x_3\}, \{x_1, c\}, \{x_2, e\}, \{x_3, \ell\} \big{\}}.
\]
Since $x_2e \in \ini_<(I)$, there exist $r,s \notin \{x_2, e\}$ such that
\[
F = x_2e - rs \in A,  \qquad  rs \notin \ini_<(I).
\] 

Now consider
\[
S(B_{12}, F) = x_1rs - eab.
\]
From (\ref{eq:red1}), $ab \notin \ini_<(I)$, and from (\ref{eq:notin}) together with $N_G(e) = \{x_2\}$ we see $ea, eb \notin \ini_<(I)$. 
Hence $\ini_<(x_1rs - eab) = x_1rs \in \ini_<(I)$. 
Since $rs \notin \ini_<(I)$, we may assume $x_1r \in \ini_<(I)$. Thus $r \in N_G(x_1) = \{x_2, x_3, c\}$. 
Because $r \neq x_2$, we must have $r \in \{x_3, c\}$. 

Next consider $S(B_{23}, F) = x_3rs - ecd.$ From (\ref{eq:red1}), $cd \notin \ini_<(I)$, and from (\ref{eq:notin}) with $N_G(e) = \{x_2\}$ we have $ec, ed \notin \ini_<(I)$. Hence $\ini_<(x_3rs - ecd) = x_3rs \in \ini_<(I)$. We have $r \in \{x_3, c\}$ but $\ini_<(I)$ is squarefree and $\{x_3, c\} \notin E(G)$, so we conclude that $x_3r \notin \ini_<(I)$. Also, we have $rs\notin \ini_<(I)$, therefore, it follows that $x_3s\in \ini_<(I)$. Again, due to $rs \notin \ini_<(I)$ and $x_3s\in \ini_<(I)$, we conclude that $r\neq x_3$, and hence $r=c$. To complete the proof, it is enough to show that the assumption $r=c$ leads to a contradiction.

Due to $x_3s\in \ini_<(I)$, we observe that $s \in N_G(x_3) = \{x_1, x_2, \ell\}$. 
Since $s \neq x_2$, and $s \neq x_1$ (because $x_1c \in \ini_<(I)$ but $cs \notin \ini_<(I)$), we obtain $s = \ell$. 
Therefore, $F = x_2e - c\ell$. This gives, $S(B_{23}, F) = c(x_3\ell - ed)$, and $x_3\ell - ed\in I$ because $I$ is prime. Note that $x_3\ell - ed\neq 0$ because $x_3\neq e$ and $x_3\neq d$ due to (\ref{eq:notin}). Also $ed\notin \ini_<(I)$ because $N_G(e)=\{x_3\}$ and $d\neq x_3$. Thus, $\ini_<(x_3\ell-ed)=x_3\ell$.

Finally, consider $S(B_{13},x_3\ell - ed)=e(x_1d-\ell f)$. This gives $x_1d-\ell f\in I$ because $I$ is prime. Moreover, $x_1d-\ell f\neq 0$ because $x_1\neq \ell$ (they are distinct vertices of $G$) and $x_1\neq f$ because of (\ref{eq:notin}). Also, $d\notin N_G(x_1)=\{x_2,x_3,c\}$ and $f\notin N_G(\ell)=\{x_3\}$ due to (\ref{eq:notin}) and (\ref{eq:neq}). This shows that $\ini_<(x_1d-\ell f)\notin \ini_<(I)$, a contradiction. This completes the proof. 
\, \, \, \, \, \, \, \, \, \, \, \, \, \, \, \, \, \,
\end{proof}

Next, we recall two classifications given in \cite{HR} for the bi--Cohen-Macaulay bipartite graphs and chordal graphs which will be useful to prove the main theorem of this section.

\begin{Theorem}\label{biCM bipartite}
{\rm (\cite[Theorem~3]{HR})}
Let $G$ be a bipartite graph with the bipartition $V(G)=V_1\cup V_2$ where $V_1=\{u_1,\ldots,u_s\}$ and $V_2=\{w_1,\ldots,w_t\}$. Then the following conditions are equivalent:
\begin{enumerate}
    \item $G$ is bi--Cohen--Macaulay;
    \item $s=t$ and $E(G)=\big{\{} \{u_i, w_j\} : 1\leq i\leq j\leq s \big{\}}$.
\end{enumerate}
\end{Theorem}

Recall that the \emph{clique complex} of a finite simple graph $G$, denoted by $\Delta(G)$, is the simplicial complex whose faces correspond to the cliques, i.e., complete subgraphs, of $G$. A \emph{free vertex} of a simplicial complex $\Delta$ is a vertex that is contained in exactly one facet of $\Delta$. 

\begin{Theorem}\label{biCM chordal}
{\rm (\cite[Theorem~4]{HR})} 
Let $G$ be a chordal graph. Then the following conditions are equivalent: 
\begin{enumerate}
    \item $G$ is bi--Cohen--Macaulay;
    \item Let $F_1,\ldots,F_m$ be the facets of $\Delta(G)$ which admit a free vertex. Then $m=1$ or $m>1$ and 
    \begin{enumerate}
        \item $V(G)$ is the disjoint union of $F_1,\ldots,F_m$; 
        \item each $F_i$ has exactly one free vertex $v_i$ for $i=1,\ldots,m$;
        \item the induced subgraph of $G$ on the vertex set $V(G)-\{v_1,\ldots,v_m\}$ is a clique. 
    \end{enumerate}
\end{enumerate}
\end{Theorem}

For a better intuition of bi--Cohen--Macaulay chordal graphs with a given number of edges, we provide the following construction.

\begin{Construction}\label{construction}
 Let $t\geq 2$ and $K_{t-1}$ be the complete graph on the vertex set $W=\{w_1,\ldots,w_{t-1}\}$. Let $W_1,\ldots,W_m$ be a partition for $W$ where $m\geq 1$. Now, let $G$ be the graph on $t-1+m$ vertices with the vertex set $W\cup \{v_1,\ldots,v_m\}$ obtained from $K_{t-1}$ by attaching the vertex $v_i$ to the vertices in $W_i$ for $i=1,\ldots, m$. More precisely, the edge set of $G$ is 
 \[
E(G)=E(K_{t-1})\cup (\bigcup_{i=1}^m \big{\{} \{v_i,w_j\} : w_j\in W_i \big{\}}). 
 \]
 Then, we have:
 \begin{enumerate}
     \item $|E(G)|={t-1\choose 2}+t-1={t\choose 2}$. 
     \item By setting $F_i=W_i\cup \{v_i\}$ for each $i=1,\ldots, m$, we see from Theorem~\ref{biCM chordal} that $G$ is a bi--Cohen--Macaulay chordal graph. 
     \item Let $H$ be a bi--Cohen--Macaulay chordal graph. Using the notation of Theorem~\ref{biCM chordal}, we set $W_i=F_i-\{v_i\}$ for each $i=1,\ldots, m$. Then, $H$ can be constructed as above.
 \end{enumerate}
\end{Construction}

Recall that $p(n)$ denotes the number of partitions of a positive integer $n$, that is, the number of ways to write $n$ as a sum of positive integers, where the order of the summands is irrelevant.

\begin{Corollary}\label{biCM chordal-2}
 A graph is a bi--Cohen--Macaulay chordal graph if and only if it is obtained as in Construction~\ref{construction}. Moreover, given $t\geq 2$, there exist exactly $p(t-1)$ non--isomorphic bi--Cohen--Macaulay chordal graphs with ${t \choose 2}$ edges.  
\end{Corollary}

\begin{proof}
   The first part is an immediate consequence of Construction~\ref{construction}~(2) and~(3). For the second part, we use the notation of Construction~\ref{construction}. Note that each pair of different partitions of the integer $|W|$ result in two non-isomorphic bi-Cohen--Macaulay chordal graphs. On the other hand, different partitions of the set $W$ which correspond to the same integer partition for $|W|$ provide isomorphic graphs in Construction~\ref{construction}. Thus, we get exactly $p(t-1)$ non--isomorphic bi--Cohen--Macaulay chordal graphs with ${t \choose 2}$ edges. \hspace{9cm}     
\end{proof}

\begin{Example}\label{biCM t=4}
 Given $t=4$, we have $p(t-1)=p(3)=3$, since $3, 2+1, 1+1+1$ are the integer partitions of~3. Thus, there are exactly three bi--Cohen--Macaulay chordal graphs with ${4\choose 2}=6$ edges by Corollary~\ref{biCM chordal-2}. These three graphs are shown in Figure~\ref{badgraphs} as $G_1, G_2$ and $G_4$.   
\end{Example}  

We conclude the present section by determining prime quadratic squarefree binomial ideals $I$ with a squarefree quadratic initial ideal and linear resolution such that $\mu(I)=6$ up to their initial ideals. In the following, a cycle of length $\ell$ is also called an $\ell$-cycle.  


\begin{Theorem}\label{6 generators}
Let $I$ be an ideal in a polynomial ring $S$ minimally generated by six quadratic squarefree binomials. Suppose that (i) $I$ is prime, (ii) $I$ has a squarefree quadratic initial ideal with respect to a monomial order~$<$ and (iii) $I$ has linear resolution. Let 
\[
\ini_<(I)=I(G) \quad \text{and} \quad \ini_{<_{lex}}(J)=I(H)
\] 
where $G$ and $H$ are finite simple graphs with no isolated vertices and $J\subset R=K[x_i,y_i: i=1,\ldots,4]$ is the determinantal ideal of the form~(\ref{determinantal}) and $<_{lex}$ is the lexicographic order induced by 
$x_1 > x_2 > x_3 > x_4 > y_1 > y_2 > y_3 > y_4$. Then $G$ and $H$ are isomorphic as finite graphs.    
\end{Theorem}

\begin{proof}
 First, note that $J$ can be seen as the binomial edge ideal of a complete graph with $n$ vertices, and hence 
 \[
 \ini_{<_{lex}}(J)=(x_iy_j : 1 \leq i < j \leq 4)\subseteq R
 \]
 by \cite[Theorem~1.1]{HHHKR}. This implies that $\ini_{<_{lex}}(J)=I(H)$ where $H$ is the bipartite graph with the vertex set 
 \[
 V(H)=\{x_1,x_2,x_3\}\cup \{y_2,y_3,y_4\}
 \] 
 and the edge set 
 \[
 E(H)=\big{\{} \{x_i,y_j\} : 1\leq i<j\leq 4 \big{\}}.
 \]
 By relabeling the vertices of $H$ with setting $z_i=y_{i+1}$ for $i=1,2,3$, it is clear that $H$ is of the type described in Theorem~\ref{biCM bipartite}.  

 On the other hand, $\ini_<(I)=I(G)$ where $G$ is a finite simple graph $G$ whose vertex is contained in the set of variables of $S$. By Lemma~\ref{initial ideal as edge ideal}, $G$ is connected and co--chordal, $S/I(G)$ is Cohen--Macaulay and $I(G)$ has linear resolution, and hence $G$ is a bi--Cohen--Macaulay graph over $K$. By Lemma~\ref{equivalency of CM+linear}~(3), we have $|E(G)|=6$. 
 
 If $G$ is not bipartite, then it has an induced odd cycle of length~$\ell\leq 5$, since $G$ has only six edges. Since the complementary graph of a 5-cycle is not chordal and since $G$ is co--chordal, it follows that $G$ does not have any induced 5-cycle. Thus, $\ell=3$. If $G$ is chordal, then $G$ is isomorphic to one of the graphs $G_1$, $G_2$ or $G_4$ depicted in Figure~\ref{badgraphs}, by Example~\ref{biCM t=4}. Then, we get a contradiction to Lemma~\ref{lemma:G1G2} and Lemma~\ref{lemma:G_4}. Thus, $G$ is not chordal and also has an 
 induced 4-cycle. The only connected graph with at least one induced 3-cycle and at least one induced 4-cycle with six edges is the graph $G_3$ shown in Figure~\ref{badgraphs}. But, according to Lemma~\ref{lemma:G_3}, $G$ is not isomorphic to $G_3$. 
  
Therefore, $G$ is a bipartite graph. By Theorem~\ref{biCM bipartite}, $G$ is isomorphic to the bipartite graph $H$ described in the first part of the proof. Hence, we get the desired result. 
\, \, \, \, \, \, \, \, \, \, \, \, \, \,
\end{proof}

We would like to remark that by the virtue of Lemma~\ref{initial ideal as edge ideal}, Theorem~\ref{possible number of generators} and Corollary~\ref{biCM chordal-2}, there are at least $p(t-1)$ potential squarefree quadratic initial ideals for the favorite class of binomial ideals $I$ in the present section with $\mu(I)={t\choose 2}$ where $t\geq 4$. Here we investigated the case $t=4$. Considering higher values of $t$ will be of great interest as well.

\section*{Acknowledgment}
Ayesha Asloob Qureshi was supported by Scientific and Technological Research Council of Turkey T\"UB\.{I}TAK under the Grant No: 124F113. Sara Saeedi Madani was in part supported by a grant from IPM (No. 1404130019).

\end{document}